\documentclass[12pt]{article}

\usepackage{geometry}
\usepackage[cp1251]{inputenc}
\usepackage[T1]{fontenc}
\usepackage[english]{babel}

\usepackage{amssymb}
\usepackage{amsthm}
\usepackage[auth-sc]{authblk}
\usepackage{cite}
\usepackage{dsfont}
\usepackage{mathrsfs}
\usepackage{mathtools}
\usepackage{tikz}
\usepackage{pgf}

\newtheorem{theorem}{Theorem}[section]
\newtheorem{lemma}[theorem]{Lemma}

\newtheorem{corollary}[theorem]{Corollary}

\providecommand{\md}{\mathop{\mathit{md}}}
\providecommand{\con}{\wedge}
\providecommand{\dis}{\vee}
\providecommand{\imp}{\to}
\providecommand{\vp}{\varphi}
\providecommand{\sameas}{\leftrightharpoons}
\providecommand{\nat}{\mathds{N}}

\newlength{\templength}
\newcommand{\predfr}[1]{\boldsymbol{\mathfrak{#1}}}

\begin{document}

\title{On algorithmic expressivity of finite-variable fragments of   intuitionistic modal logics\thanks{The work on this paper, carried out at the Institute for Information Transmission Problems of the Russian Academy of Sciences, has been supported by Russian Science Foundation, Project~\mbox{21-18-00195}.}}

%

\author[1]{Mikhail Rybakov}
\author[2]{Dmitry Shkatov}

\affil[1]{IITP RAS, HSE University, Tver State University}
\affil[2]{University of the Witwatersrand, Johannesburg}


\date{}


\maketitle

\section{Introduction}

Modal and intuitionistic propositional logics are often poly-time
embeddable into their own fragments with a few variables (typically,
zero, one, or two), and similar embeddings are sometimes constructed
of fragments of logics with special properties into finite-variable
fragments of those logics.  The literature on the topic is quite
extensive~\cite{BS93,Spaan93,Halpern95,Hemaspaandra01,DS02,Sve03} and
includes contributions by the authors of this
paper~\cite{ChRyb03,Rybakov04,Rybakov06,Rybakov07a,Rybakov07,Rybakov08,RShIGPL18,RShICTAC18,RShSaicsit18,RShIGPL19,RShJLC21a,RShJLC22,RShTCS22}.

As a result, the validity problem for such fragments is as
computationally hard as the validity problem for the full logic.  (In
general, modal and superintuitionistic propositional logics, even
linearly approximable ones, may have arbitrarily hard fragments with a
few variables since, for every set $A \subseteq \mathds{N}$, one can
construct~\cite{RShJLC23} a linearly approximable logic whose fragment
with a few variables (typically zero, one, or two) recursively encodes
$A$. We obtain here similar embeddings for the intuitionistic modal
logics $\mathbf{FS}$ and $\mathbf{MIPC}$, introduced by, respectively,
Fisher~Servi~\cite{FS77} and Prior~\cite{Prior57}.  These logics have
been introduced as counterparts of bimodal propositional logics, and
can also be viewed as fragments of the predicate intuitionistic logic
$\mathbf{QInt}$ (for details, see~\cite{GKWZ}); we note that this is
not the only approach to constructing modal intuitionistic logics,
cf.~\cite{Dosen84,Dosen86,Speranski}.  The complexity of $\mathbf{FS}$
and $\mathbf{MIPC}$ remains unresolved, but the results presented here
show that single-variable fragments of these logics have the same
complexity as the full logics.

\section{Preliminaries}
\label{sec:prelim}

The intuitionistic modal language contains a countable set
$\mathcal{P}$ of propositional variables, the constant $\bot$, binary
connectives $\con$, $\dis$, and $\imp$, and unary modal connectives
$\Diamond$ and $\Box$. Formulas are defined in the usual way.  A
formula is \textit{positive} if it does not contain occurrences of
$\bot$. The set of propositional variables of a formula $\vp$ is
denoted by $\mathop{\mathit{var}} \vp$. The result of substituting a
formula $\psi$ for a variable $p$ into a formula $\vp$ is denoted by
$[\psi / p] \vp$. The modal depth of a formula $\vp$, denoted by
$\md \vp$, is the maximal number of nested modal connectives in
$\vp$. The length of a formula $\vp$, defined as the number of symbols
in $\vp$ (with the binary encoding of variables), is denoted by
$|\vp|$.

We define the logics $\mathbf{FS}$ and $\mathbf{MIPC}$ semantically.
A \textit{Kripke frame} is a pair
$\mathfrak{F} = \langle W, R \rangle$ where $W$ is a non-empty set of
\textit{worlds} and $R$ is a partial order on $W$.  An
\textit{$\mathbf{FS}$-frame} is a triple
$\predfr{F} = \langle W, R, \delta \rangle$, where
$\langle W, R \rangle$ is a Kripke frame and $\delta$ is a map
associating with each $w \in W$ a structure
$\langle \Delta_w, S_w \rangle$, with $\Delta_w$ being a non-empty set
of \textit{points} and $S_w$ a binary relation on $\Delta_w$ such
that, for every $w, v \in W$,
$$
\begin{array}{lcl}
  v \in R(w) ~\Rightarrow~ \Delta_w \subseteq \Delta_v
  & \mbox{ and }
  & \mbox{$S_w \subseteq S_v$}
\end{array}
$$  
An
\mbox{$\mathbf{FS}$-frame} $\predfr{F} = \langle W, R, \delta \rangle$
is an \textit{\mbox{$\mathbf{MIPC}$-frame}} if
$S_w = \Delta_w \times \Delta_w$, for every $w \in W$.  A
\textit{valuation} on an $\mathbf{FS}$-frame
$\langle W, R, \delta \rangle$ is a map associating with each
$w \in W$ and each $p \in \mathcal{P}$ a subset $V(w, p)$ of
$\Delta_w$ in such a way that
$$
v \in R(w) ~\Rightarrow~  V(w, p) \subseteq V(v, p).
$$
The pair $\mathfrak{M} = \langle \predfr{F}, V \rangle$, where
$\predfr{F}$ is an $\mathbf{FS}$-frame and $V$ a valuation on
$\predfr{F}$, is called an \textit{$\mathbf{FS}$-model}.  An
\textit{$\mathbf{MIPC}$-model} is an $\mathbf{FS}$-model over an
$\mathbf{MIPC}$-frame.  The \textit{truth-relation $\models$} is
defined by recursion (here, $\mathfrak{M}$ is a model, $w \in W$,
$x \in \Delta_w$, and $\vp$ is a formula):
\settowidth{\templength}{\mbox{$\mathfrak{M}, w, x \models \vp_1 \imp
    \vp_2$}}
\begin{itemize}
\item
\mbox{\parbox{\templength}{$\mathfrak{M},  w, x  \models p$}
           \mbox{~~$\sameas$~~}
           \mbox{$x \in V(w, p)$ \qquad if $p \in \mathcal{P}$;}}
\item
\mbox{$\mathfrak{M},  w, x  \not\models \bot$;}
\item
\mbox{\parbox{\templength}{$\mathfrak{M},  w, x  \models \vp_1 \con \vp_2$}
           \mbox{~~$\sameas$~~}
           \mbox{$\mathfrak{M},  w, x  \models \vp_1$
             and $\mathfrak{M},  w, x  \models \vp_2$;}}
\item
\mbox{\parbox{\templength}{$\mathfrak{M},  w, x  \models \vp_1 \dis \vp_2$}
           \mbox{~~$\sameas$~~}
           \mbox{$\mathfrak{M},  w, x  \models \vp_1$
             or $\mathfrak{M},  w, x  \models \vp_2$;}}
\item
\mbox{\parbox{\templength}{$\mathfrak{M},  w, x  \models \vp_1 \imp \vp_2$}
           \mbox{~~$\sameas$~~}
           \mbox{$\mathfrak{M},  v, x  \not\models \vp_1$
             or $\mathfrak{M},  v, x  \models \vp_2$
             whenever $v \in R(w)$;}}
\item
\mbox{\parbox{\templength}{$\mathfrak{M},  w, x
              \models \Diamond \vp_1$} \mbox{~~$\sameas$~~}
           \mbox{$\mathfrak{M},  w, y  \models \vp_1$,
             for some $y \in S_w(x)$;}}
\item
\mbox{\parbox{\templength}{$\mathfrak{M},  w, x
              \models \Box \vp_1$} \mbox{~~$\sameas$~~}
           \mbox{$\mathfrak{M},  v, y  \models \vp_1$
             whenever $v \in R(w)$ and $y \in S_v(x)$.}}

\end{itemize}

A formula $\vp$ is \textit{true} in a model $\mathfrak{M}$ (notation:
$\mathfrak{M} \models \vp$) if $\mathfrak{M}, w, x \models \vp$, for
every world $w$ of $\mathfrak{M}$ and every point $x$ of $w$.  A
formula $\vp$ is \textit{valid} an $\mathbf{FS}$-frame $\mathfrak{F}$
if $\vp$ is true in every model over $\mathfrak{F}$.  Logics
$\mathbf{FS}$ and $\mathbf{MIPC}$ are defined as sets of formulas
valid on, respectively, every $\mathbf{FS}$-frame and every
$\mathbf{MIPC}$-frame.

\section{Main results}
\label{sec:embedding}

In this section, we prove that logics $\mathbf{FS}$ and
$\mathbf{MIPC}$ are polynomial-time embeddable into their own
fragments with a single propositional variable.  We first poly-time
embed these logics into their own positive fragments. Let $\vp$ be a
formula and $f \in \mathcal{P} \setminus \mathop{\mathit{var}} \vp$.
Define
$$
\begin{array}{llll}
  
  \vp^f  =   [f / \bot]\vp; &  F_1  =  \Diamond^{\leqslant \md \vp}
                                f \imp f; & 
                                              F_2  =  f \imp \Box^{\leqslant \md \vp} f; & 
                                                                                             \displaystyle    F_3  =  \bigwedge\limits_{\mathclap{p \in \mathop{\mathit{var}} \vp}}
                                                                                             \Box^{\leqslant \md \vp} ( f \imp p ),
\end{array}
$$
and put $F  =  F_1 \con F_2 \con F_3$.
\begin{lemma}
  \label{lem:positive}
  Let $\vp$ be a formula,
  $f \in \mathcal{P} \setminus \mathop{\mathit{var}} \vp$, and
  $L \in \{\mathbf{FS}, \mathbf{MIPC}\}$.  Then,
  $$
  \begin{array}{lcl}
    \vp \in L & \iff & F \imp \vp^f \in L.
  \end{array}          
  $$
\end{lemma}

Since $\vp^f$ and $F$ are both positive, the map
$e \colon \vp \mapsto (F \imp \vp^f)$ embeds $\mathbf{FS}$ and
$\mathbf{MIPC}$ into their own positive fragments.

We next define a polytime computable function $\cdot^\ast$ from the
set of positive formulas to the set of one-variable positive formulas
and show that, for $L \in \{ \mathbf{FS}, \mathbf{MIPC} \}$ and every
positive $\vp$,
$$
\begin{array}{lcl}
  \vp^\ast \in L & \iff & \vp \in L.
\end{array}
$$
Hence, for every $\vp$,
$$
\begin{array}{lclcl}
  \vp \in L & \iff & e(\vp) \in L & \iff & e(\vp)^\ast \in L.
\end{array}
$$

The formula $\vp^\ast$ shall be obtain from $\vp$ using a
substitution.  We next define the formulas that shall be substituted
for propositional variables of $\vp$. These formulas, except $G_1$,
$G_2$, and $G_3$, are divided into `levels', indexed by elements
of~$\mathds{N}$; formulas of level $0$ are denoted $A^0_i$ or $B^0_i$,
those of level $1$, by $A^1_i$ and $B^1_i$, etc.  We begin with $G_1$,
$G_2$, and $G_3$, as well as formulas of levels~$0$ and~$1$:
$$
\begin{array}{lclclcl}
  G_1 & = & \Diamond p; & A_1^1 & = & A_1^0 \con A_2^0 \imp B_1^0 \dis B_2^0;   \\
  G_2 & = & \Diamond p \imp p; &  A_2^1 & = & A_1^0 \con
                                                             B_1^0 \imp
                                                             A_2^0 \dis
                                                             B_2^0; \\
  G_3 & = & p \imp \Box p; & A_3^1 & =
                                       & A_1^0 \con B_2^0
                                         \imp A_2^0 \dis
                                         B_1^0;  \\
  A_1^0 & = & G_2 \imp G_1 \dis G_3; & B_1^1& = & A_2^0 \con B_1^0 \imp A_1^0 \dis B_2^0; \\
  A_2^0 & = & G_3 \imp G_1 \dis G_2; & B_2^1& = & A_2^0 \con B_2^0 \imp A_1^0 \dis B_1^0; \\
  B_1^0 & = & G_1 \imp G_2 \dis G_3; & B_3^1& = & B_1^0 \con B_2^0 \imp A_1^0 \dis A_2^0. \\
  B_2^0 & = & A_1^0 \con A_2^0 \con B_1^0 \imp G_1 \dis G_2 \dis G_3;  \\
\end{array}
$$
We proceed by recursion.  Let $k \geqslant 1$.  Suppose the formulas
$A^k_1, \ldots, A^k_{n_k}$ and $B^k_1, \ldots, B^k_{n_k}$ have been
defined, with $n_k$ being the number of formulas of the form $A^k_i$
and, also, the number of formulas of the form $B^k_i$ (e.g., if
$k = 1$, then $n_k = 3$; the recursive definition for the cases where
$k \geqslant 2$ is to be given).  Define a linear order $\prec$ on the
set $(\nat \setminus \{0, 1\}) \times (\nat \setminus \{0, 1\})$ as in
the following picture, so that
$\langle i, j \rangle \prec \langle i', j' \rangle$ if, and only if,
there exists a path along one or more arrows from
$\langle i, j \rangle $ to $\langle i', j' \rangle$:

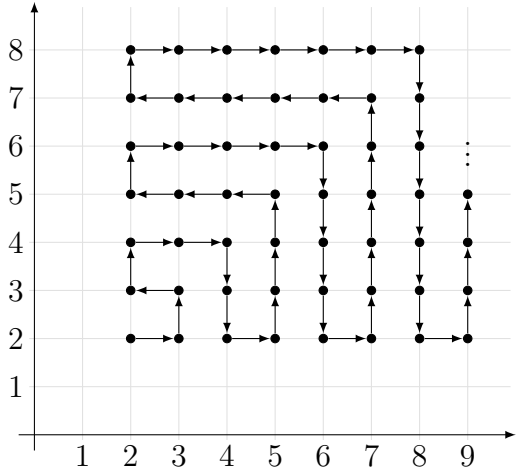
\begin{center}
\begin{tikzpicture}[scale=0.64]

\coordinate (g00)   at (0,0);
\coordinate (g10)   at (1,0);
\coordinate (g20)   at (2,0);
\coordinate (g30)   at (3,0);
\coordinate (g40)   at (4,0);
\coordinate (g50)   at (5,0);
\coordinate (g60)   at (6,0);
\coordinate (g70)   at (7,0);
\coordinate (g80)   at (8,0);
\coordinate (g01)   at (0,1);
\coordinate (g11)   at (1,1);
\coordinate (g21)   at (2,1);
\coordinate (g31)   at (3,1);
\coordinate (g41)   at (4,1);
\coordinate (g51)   at (5,1);
\coordinate (g61)   at (6,1);
\coordinate (g71)   at (7,1);
\coordinate (g81)   at (8,1);
\coordinate (g91)   at (9,1);
\coordinate (g02)   at (0,2);
\coordinate (g12)   at (1,2);
\coordinate (g22)   at (2,2);
\coordinate (g32)   at (3,2);
\coordinate (g42)   at (4,2);
\coordinate (g52)   at (5,2);
\coordinate (g62)   at (6,2);
\coordinate (g72)   at (7,2);
\coordinate (g82)   at (8,2);
\coordinate (g92)   at (9,2);
\coordinate (g03)   at (0,3);
\coordinate (g13)   at (1,3);
\coordinate (g23)   at (2,3);
\coordinate (g33)   at (3,3);
\coordinate (g43)   at (4,3);
\coordinate (g53)   at (5,3);
\coordinate (g63)   at (6,3);
\coordinate (g73)   at (7,3);
\coordinate (g83)   at (8,3);
\coordinate (g93)   at (9,3);
\coordinate (g04)   at (0,4);
\coordinate (g14)   at (1,4);
\coordinate (g24)   at (2,4);
\coordinate (g34)   at (3,4);
\coordinate (g44)   at (4,4);
\coordinate (g54)   at (5,4);
\coordinate (g64)   at (6,4);
\coordinate (g74)   at (7,4);
\coordinate (g84)   at (8,4);
\coordinate (g94)   at (9,4);
\coordinate (g05)   at (0,5);
\coordinate (g15)   at (1,5);
\coordinate (g25)   at (2,5);
\coordinate (g35)   at (3,5);
\coordinate (g45)   at (4,5);
\coordinate (g55)   at (5,5);
\coordinate (g65)   at (6,5);
\coordinate (g75)   at (7,5);
\coordinate (g85)   at (8,5);
\coordinate (g95)   at (9,5);
\coordinate (g06)   at (0,6);
\coordinate (g16)   at (1,6);
\coordinate (g26)   at (2,6);
\coordinate (g36)   at (3,6);
\coordinate (g46)   at (4,6);
\coordinate (g56)   at (5,6);
\coordinate (g66)   at (6,6);
\coordinate (g76)   at (7,6);
\coordinate (g86)   at (8,6);
\coordinate (g96)   at (9,6);
\coordinate (g07)   at (0,7);
\coordinate (g17)   at (1,7);
\coordinate (g27)   at (2,7);
\coordinate (g37)   at (3,7);
\coordinate (g47)   at (4,7);
\coordinate (g57)   at (5,7);
\coordinate (g67)   at (6,7);
\coordinate (g77)   at (7,7);
\coordinate (g87)   at (8,7);
\coordinate (g08)   at (0,8);
\coordinate (g18)   at (1,8);
\coordinate (g28)   at (2,8);
\coordinate (g38)   at (3,8);
\coordinate (g48)   at (4,8);
\coordinate (g58)   at (5,8);
\coordinate (g68)   at (6,8);
\coordinate (g78)   at (7,8);
\coordinate (g88)   at (8,8);
\coordinate (g09)   at (0,9);
\coordinate (g19)   at (1,9);
\coordinate (g29)   at (2,9);
\coordinate (g39)   at (3,9);
\coordinate (g49)   at (4,9);
\coordinate (g59)   at (5,9);
\coordinate (g69)   at (6,9);
\coordinate (g79)   at (7,9);
\coordinate (g89)   at (8,9);
\coordinate (g90)   at (9,0);
\coordinate (g99)   at (9,9);
\coordinate (g010)  at (0,10);
\coordinate (g110)  at (1,10);
\coordinate (g210)  at (2,10);
\coordinate (g310)  at (3,10);
\coordinate (g410)  at (4,10);
\coordinate (g510)  at (5,10);
\coordinate (g610)  at (6,10);
\coordinate (g710)  at (7,10);
\coordinate (g810)  at (8,10);
\coordinate (g910)  at (9,10);
\coordinate (g100)  at (10,0);
\coordinate (g101)  at (10,1);
\coordinate (g102)  at (10,2);
\coordinate (g103)  at (10,3);
\coordinate (g104)  at (10,4);
\coordinate (g105)  at (10,5);
\coordinate (g106)  at (10,6);
\coordinate (g107)  at (10,7);
\coordinate (g108)  at (10,8);

\begin{scope}[>=latex, -, shorten >= 2pt, shorten <= -1.96pt, color=black!12]
\draw [] (g01) -- (g101);
\draw [] (g02) -- (g102);
\draw [] (g03) -- (g103);
\draw [] (g04) -- (g104);
\draw [] (g05) -- (g105);
\draw [] (g06) -- (g106);
\draw [] (g07) -- (g107);
\draw [] (g08) -- (g108);
\draw [] (g10) -- (g19);
\draw [] (g20) -- (g29);
\draw [] (g30) -- (g39);
\draw [] (g40) -- (g49);
\draw [] (g50) -- (g59);
\draw [] (g60) -- (g69);
\draw [] (g70) -- (g79);
\draw [] (g80) -- (g89);
\draw [] (g90) -- (g99);
\end{scope}

\begin{scope}[>=latex, ->, shorten >= 0pt, shorten <= -5.96pt]
\draw [] (g00) -- (g100);
\draw [] (g00) -- (g09);
\end{scope}

\node [below] at (g10) {$1$};
\node [below] at (g20) {$2$};
\node [below] at (g30) {$3$};
\node [below] at (g40) {$4$};
\node [below] at (g50) {$5$};
\node [below] at (g60) {$6$};
\node [below] at (g70) {$7$};
\node [below] at (g80) {$8$};
\node [below] at (g90) {$9$};
\node [left ] at (g01) {$1$};
\node [left ] at (g02) {$2$};
\node [left ] at (g03) {$3$};
\node [left ] at (g04) {$4$};
\node [left ] at (g05) {$5$};
\node [left ] at (g06) {$6$};
\node [left ] at (g07) {$7$};
\node [left ] at (g08) {$8$};

\begin{scope}[>=latex, ->, shorten >= 1.96pt, shorten <= 1.96pt]
\draw [] (g22) -- (g32);
\draw [] (g32) -- (g33);
\draw [] (g33) -- (g23);
\draw [] (g23) -- (g24);
\draw [] (g24) -- (g34);
\draw [] (g34) -- (g44);
\draw [] (g44) -- (g43);
\draw [] (g43) -- (g42);
\draw [] (g42) -- (g52);
\draw [] (g52) -- (g53);
\draw [] (g53) -- (g54);
\draw [] (g54) -- (g55);
\draw [] (g55) -- (g45);
\draw [] (g45) -- (g35);
\draw [] (g35) -- (g25);
\draw [] (g25) -- (g26);
\draw [] (g26) -- (g36);
\draw [] (g36) -- (g46);
\draw [] (g46) -- (g56);
\draw [] (g56) -- (g66);
\draw [] (g66) -- (g65);
\draw [] (g65) -- (g64);
\draw [] (g64) -- (g63);
\draw [] (g63) -- (g62);
\draw [] (g62) -- (g72);
\draw [] (g72) -- (g73);
\draw [] (g73) -- (g74);
\draw [] (g74) -- (g75);
\draw [] (g75) -- (g76);
\draw [] (g76) -- (g77);
\draw [] (g77) -- (g67);
\draw [] (g67) -- (g57);
\draw [] (g57) -- (g47);
\draw [] (g47) -- (g37);
\draw [] (g37) -- (g27);
\draw [] (g27) -- (g28);
\draw [] (g28) -- (g38);
\draw [] (g38) -- (g48);
\draw [] (g48) -- (g58);
\draw [] (g58) -- (g68);
\draw [] (g68) -- (g78);
\draw [] (g78) -- (g88);
\draw [] (g88) -- (g87);
\draw [] (g87) -- (g86);
\draw [] (g86) -- (g85);
\draw [] (g85) -- (g84);
\draw [] (g84) -- (g83);
\draw [] (g83) -- (g82);
\draw [] (g82) -- (g92);
\draw [] (g92) -- (g93);
\draw [] (g93) -- (g94);
\draw [] (g94) -- (g95);
\end{scope}

\node [] at (g96) {$\vdots$};

\filldraw [] (g22) circle [radius=2.5pt]   ;
\filldraw [] (g32) circle [radius=2.5pt]   ;
\filldraw [] (g42) circle [radius=2.5pt]   ;
\filldraw [] (g52) circle [radius=2.5pt]   ;
\filldraw [] (g62) circle [radius=2.5pt]   ;
\filldraw [] (g72) circle [radius=2.5pt]   ;
\filldraw [] (g82) circle [radius=2.5pt]   ;
\filldraw [] (g92) circle [radius=2.5pt]   ;
\filldraw [] (g23) circle [radius=2.5pt]   ;
\filldraw [] (g33) circle [radius=2.5pt]   ;
\filldraw [] (g43) circle [radius=2.5pt]   ;
\filldraw [] (g53) circle [radius=2.5pt]   ;
\filldraw [] (g63) circle [radius=2.5pt]   ;
\filldraw [] (g73) circle [radius=2.5pt]   ;
\filldraw [] (g83) circle [radius=2.5pt]   ;
\filldraw [] (g93) circle [radius=2.5pt]   ;
\filldraw [] (g24) circle [radius=2.5pt]   ;
\filldraw [] (g34) circle [radius=2.5pt]   ;
\filldraw [] (g44) circle [radius=2.5pt]   ;
\filldraw [] (g54) circle [radius=2.5pt]   ;
\filldraw [] (g64) circle [radius=2.5pt]   ;
\filldraw [] (g74) circle [radius=2.5pt]   ;
\filldraw [] (g84) circle [radius=2.5pt]   ;
\filldraw [] (g94) circle [radius=2.5pt]   ;
\filldraw [] (g25) circle [radius=2.5pt]   ;
\filldraw [] (g35) circle [radius=2.5pt]   ;
\filldraw [] (g45) circle [radius=2.5pt]   ;
\filldraw [] (g55) circle [radius=2.5pt]   ;
\filldraw [] (g65) circle [radius=2.5pt]   ;
\filldraw [] (g75) circle [radius=2.5pt]   ;
\filldraw [] (g85) circle [radius=2.5pt]   ;
\filldraw [] (g95) circle [radius=2.5pt]   ;
\filldraw [] (g26) circle [radius=2.5pt]   ;
\filldraw [] (g36) circle [radius=2.5pt]   ;
\filldraw [] (g46) circle [radius=2.5pt]   ;
\filldraw [] (g56) circle [radius=2.5pt]   ;
\filldraw [] (g66) circle [radius=2.5pt]   ;
\filldraw [] (g76) circle [radius=2.5pt]   ;
\filldraw [] (g86) circle [radius=2.5pt]   ;
\filldraw [] (g27) circle [radius=2.5pt]   ;
\filldraw [] (g37) circle [radius=2.5pt]   ;
\filldraw [] (g47) circle [radius=2.5pt]   ;
\filldraw [] (g57) circle [radius=2.5pt]   ;
\filldraw [] (g67) circle [radius=2.5pt]   ;
\filldraw [] (g77) circle [radius=2.5pt]   ;
\filldraw [] (g87) circle [radius=2.5pt]   ;
\filldraw [] (g28) circle [radius=2.5pt]   ;
\filldraw [] (g38) circle [radius=2.5pt]   ;
\filldraw [] (g48) circle [radius=2.5pt]   ;
\filldraw [] (g58) circle [radius=2.5pt]   ;
\filldraw [] (g68) circle [radius=2.5pt]   ;
\filldraw [] (g78) circle [radius=2.5pt]   ;
\filldraw [] (g88) circle [radius=2.5pt]   ;



\end{tikzpicture}
\end{center}

We can then define an enumeration $g$ of the pairs
$\langle i, j \rangle \in (\nat \setminus \{0, 1\}) \times (\nat
\setminus \{0, 1\})$ according to $\prec$, i.e., so that
$g(2, 2) = 1$, $g(3, 2) = 2$, $g(3, 3) = 3$, $g(2, 3) = 4$, etc.  Now,
for every $i, j \in \{2, \ldots, n_k \}$, define
$$
\begin{array}{lclclcl}
  A_{g(i,j)}^{k+1} & = & A_1^k \imp B_1^k \dis A_i^k \dis B_j^k; & &
  B_{g(i,j)}^{k+1} & = & B_1^k \imp A_1^k \dis A_i^k \dis B_j^k,
\end{array}
$$
and let $n_{k+1}$ be the number of the formulas of the form
$A_i^{k+1}$ (which is the same as the number of formulas of the form
$B_i^{k+1}$) so defined; notice that $n_{k+1} = (n_k - 1)^2$.  This
completes the recursive definition of $A^k_i$ and $B^k_i$.

Next, put
$$
\begin{array}{lcl}
  l_0 & = & |A^0_1| + |B^0_1| + |A^0_2| + |B^0_2|.
\end{array}            
$$

\begin{lemma}
  \label{lem:inequ-2}
  There exists $k_0\in \mathds{N}$ such that $n_k > l_0 \cdot 5^k$
  whenever $k \geqslant k_0$.
\end{lemma}

Now, let $\vp$ be a positive formula with
$\mathop{\mathit{var}} \vp = \{p_1, \ldots, p_s \}$.  Let ${k_\vp}$ be
the least integer $k$ such that $|\vp| < l_0 \cdot 5^{k}$.  By
Lemma~\ref{lem:inequ-2},
$n_{{k_\vp}+k_0} > l_0 \cdot 5^{{k_\vp}+k_0}$; hence,
$$
n_{{k_\vp}+k_0} > l_0 \cdot 5^{{k_\vp}+k_0} > 5^{k_0} \cdot |\vp| >
|\vp| \geqslant s.
$$

Lastly, define $\vp^\ast$ to be the result of substituting into $\vp$,
for every $r \in \{1, \ldots, s\}$, the formula
$A_r^{{k_\vp}+k_0} \dis B_r^{{k_\vp}+k_0}$ for the variable $p_r$
(this substitution is well defined since
$n_{{k_\vp}+k_0} > s$).

\medskip

We next show that $\vp^\ast$ is poly-time computable from $\vp$.

\begin{lemma}
  \label{lem:inequ-1}
  For every $k \geqslant 0$ and every $i \in \{1, \ldots, n_k\}$,
  $$
  \begin{array}{lcl}
    |A^k_i| < l_0 \cdot 5^k & \mbox{and} & |B^k_i| < l_0 \cdot 5^k.
  \end{array}
  $$
\end{lemma}

\begin{lemma}
  \label{lem:main-lemma-2}
  The formula $\varphi^\ast$ is computable in time polynomial
  in~$|\varphi|$.
\end{lemma}

\begin{proof}
  It suffices to show that $|\vp^\ast|$ is polynomial in $|\vp|$.
  Since ${k_\vp}$ is the least integer $k$ such that
  $|\vp| < l_0 \cdot 5^{k}$, surely
  $l_0 \cdot 5^{k_\vp - 1} \leqslant |\vp|$, and so
  $$
  l_0 \cdot 5^{k_\vp + k_0} \leqslant 5^{k_0 + 1} |\vp|.
  $$
  By Lemma~\ref{lem:inequ-1}, for every $i \in \{1, \ldots, n_{k_\vp + k_0}\}$,
  $$
  \begin{array}{lcl}
    |A^{k_\vp + k_0}_i| < l_0 \cdot 5^{k_\vp + k_0} \leqslant 5^{k_0 +
    1} |\vp| & \mbox{and} & |B^{k_\vp + k_0}_i| < l_0 \cdot 5^{k_\vp +
                            k_0} \leqslant 5^{k_0 + 1} |\vp|.
  \end{array}
  $$
  Hence, $|\vp^\ast| < 2 \cdot 5^{k_0 + 1} |\vp|^2$. 
\end{proof}

To obtain the desired result, it remains to show the following:

\begin{lemma}
  \label{lem:main-lemma}
  Let $L \in \{ \mathbf{FS}, \mathbf{MIPC} \}$.  Then, for every
  positive formula $\vp$,
  $$
  \begin{array}{lcl}
    \vp \in L & \iff & \vp^\ast \in L.
  \end{array}  
  $$
\end{lemma}

From Lemmas~\ref{lem:positive}, \ref{lem:main-lemma-2},
and~\ref{lem:main-lemma}, we immediately obtain the following:

\begin{theorem}
  \label{thr:main}
  Let $L \in \{ \mathbf{FS}, \mathbf{MIPC} \}$.  Then, there exists a
  polynomial-time computable function embedding $L$ into its own
  positive one-variable fragment.
\end{theorem}

\begin{corollary}
  \label{cor:int-main}
  Let $L \in \{ \mathbf{FS}, \mathbf{MIPC} \}$.  Then, the positive
  one-variable fragment of $L$ is polytime-equivalent to~$L$.
\end{corollary}

The results presented here are not immediately applicable to obtaining
the computational complexity of finite-variable fragments of
intuitionistic modal logics since the complexity of full logics
remains unknown (we are only aware of decidability
results~\cite{Grefe98,WZ99,WZ99IM,AlSh06,Gir} for modal intuitionistic logics).





\end{document}